\documentclass{amsart}
\usepackage{amssymb,amsmath}
\usepackage{amsfonts}
\usepackage{color,cite}
\usepackage[normalem]{ulem}

\newtheorem{theorem}{Theorem}
\theoremstyle{plain}

\newtheorem{corollary}{Corollary}

\newtheorem{lemma}{Lemma}

\newtheorem{remark}{Remark}

\numberwithin{equation}{section}

\input{tcilatex}

\begin{document}
\title{Uniqueness of one-dimensional N\'{e}el wall {profiles}}
\author{Cyrill B. Muratov}
\address{Department of Mathematics, 
New Jersey Institute of Technology, 
Newark, NJ 07102, USA}
\email{muratov@njit.edu}
\urladdr{http://m.njit.edu/\symbol{126}muratov}
\thanks{}
\author{Xiaodong Yan}
\address{Department of Mathematics, 
University of Connecticut, 
Storrs, CT 06269}
\email{xiaodong.yan@uconn.edu}
\urladdr{http://www.math.uconn.edu/\symbol{126}xiaodong}
\thanks{X.Y. would like to thank Christof Melcher for helpful discussions. {%
The work of C.B.M. was supported, in part, by NSF via grants DMS-0908279 and
DMS-1313687.}}
\date{\today }

\begin{abstract}
  We study the domain wall structure in thin uniaxial ferromagnetic
  films in the presence of an in-plane applied external field in the
  direction normal to the easy axis. Using the reduced one-dimensional
  thin film micromagnetic model, we {analyze} the critical points of
  the obtained non-local variational problem. We prove {that the
    minimizer of the one-dimensional energy functional in the form of
    the N\'{e}el wall is the unique (up to translations) critical
    point of the energy among all monotone profiles with the same
    limiting behavior at infinity. Thus, we establish} uniqueness of
  the one-dimensional monotone N\'{e}el wall profile in {the
    considered} setting. We also obtain some uniform estimates for
  general one-dimensional domain wall profiles.
\end{abstract}

\maketitle

\section{\protect\bigskip Introduction}

\bigskip

{Thin soft} ferromagnetic films have been widely used as a data storage
solution in modern computer technology {{\cite{HS,Moser,eleftheriou10}}}. It
is well established that for sufficiently thin films, the magnetization
vector of the material lies almost entirely in the film plane. {Such
ultra-thin ferromagnetic films often} exhibit magnetization patterns
consisting of domains in which the magnetization vector is nearly constant {%
and is aligned along one of the directions of the material's easy axis.}
Domains {with different orientation of the magnetization are separated} by
thin transition layers called \emph{\ domain walls}, in which the
magnetization vector {rotates} rapidly from one direction to another.

The study of the domain wall structure in ferromagnetics materials has
attracted a lot of attention. One of the common domain wall types in
ultrathin ferromagnetic films is the \emph{N\'{e}el wall}. In this wall
type, the magnetization vector exhibits an in-plane 180$^{\circ }$ rotation
in the absence of an applied magnetic field. At present, the structure of
the N\'{e}el wall is rather well understood. Within the framework of
micromagnetic modeling, the overall physical picture has been summarized in
books \cite{Ah} and \cite{HS} (see also \cite{GC1,Mi,RS,MO}, etc.).
Experimental {observations} of the one-dimensional N\'{e}el wall profiles
can be found in \cite{BO, JAB,WL}. Rigorous mathematical analysis of N\'{e}%
el wall is more recent, starting from the work of Garc\'{\i}a-Cervera on the
analysis of the associated one-dimensional variational problem \cite{GC,GC1}%
. Melcher studied one-dimensional energy minimizers in thin uniaxial films
and obtained symmetry, monotonicity of the one-dimensional minimizing
profile, as well as the logarithmic decay beyond the core region for very
soft films \cite{Mel}. Linearized stability of the one-dimensional N\'{e}el
wall with respect to one-dimensional perturbations {in a reduced thin film
model was} proved in \cite{CMO}. Asymptotic stability of one-dimensional N%
\'{e}el walls with respect to large two-dimensional perturbations {in a
reduced two-dimensional thin film model was} demonstrated in \cite{DKO}.

Recently, Chermisi and Muratov studied the reduced one-dimensional energy in
the presence of an applied in-plane magnetic field in the direction
perpendicular to the easy axis \cite{CM}. They expressed the magnetic energy
in terms of the phase angle rather than the usual two-dimensional unit
vector representation of the magnetization. They obtained uniqueness and
strict monotonicity of the angle variable for the minimizing N\'{e}el wall
structure. Moreover, they proved precise asymptotic behavior of the
minimizing N\'{e}el wall profiles at infinity{{\footnote{{{We point out that
the proof of the asymptotic decay of the N\'{e}el wall profiles in \cite{CM}
contained an error, which, however, does not affect the result. For the
reader's convenience, we present the correction in the appendix.}}}}}. The
associated Euler-Lagrange equation in their setting is expressed as an
ordinary differential equation for the phase angle with a nonlocal term
present.

We note that while from the physical point of view the N\'{e}el walls are
believed to be the energy minimizing configurations of the magnetization
connecting the two oppositely oriented domains in uniaxial films, it is
natural to ask whether other, metastable N\'{e}el wall-type configurations
connecting the two domains, are also possible. For example, in the presence
of a transverse in-plane magnetic field one can distinguish normal and
reverse domain walls, which differ by the rotation sense of the
magnetization \cite{PH}. Clearly, the reverse domain wall is not an energy
minimizer, since the magnetization vector opposes the applied field in such
a wall. Still, in view of the highly nonlinear and non-local character of
the problem it is not a priori clear whether there could exist other
one-dimensional domain wall profiles connecting the domains of opposing
magnetization which are only local, but not global minimizers of the
micromagnetic energy.

In this paper, we follow the variational setting introduced in
\cite{CM} and consider the \emph{critical points} of the associated
energy functional which are monotone {in the angle variable}. We prove
that any monotone critical point of the reduced one-dimensional energy
is unique (up to translations) and, therefore, is the
  minimizer. Thus, we establish that monotone one-dimensional
  magnetization profiles that are not global energy minimizers do not
  exist, corroborating the expected physical picture. This also
provides a better understanding of the {results of the} numerical
solution of the considered problem and allows to conclude that the
obtained one-dimensional profiles (see, e.g., \cite{MO}) indeed
correspond to the N\'eel walls. In addition, we address the
  question of uniform regularity of the critical points of the
  one-dimensional energy and establish uniform bounds and, hence,
  decay of all the derivatives of such solutions at infinity. This
  result can be applied to other types of domain wall profiles of
  interest, such as those of the 360$^\circ$ walls
  \cite{mo:jap08,KKM}.

The rest of this paper is organized as follows: In section two, we
recall some basic facts about the micromagnetic energy and the reduced
one-dimensional energy in the presence of an applied in-palne field
oriented normally to the easy axis. The main results are stated
at the end of section two. The proof of the uniqueness theorem is
presented in sections three, and the proof of the uniform
  estimates for the derivaties of domain wall solutions is given in
  section four. Finally, we briefly revisit the question of the decay
of N\'eel walls at infinity in the Appendix.

\section{Variational {setting} and statement of the main result}

In this paper, we are interested in the analysis of magnetization
configurations in thin uniaxial ferromagnetic films of large extent with in
plane easy axis and applied in-plane field normal to the easy axis. The
energy functional related to such a system, introduced by Landau and
Lifschitz, can be written in CGS\ units as a combination of five terms: 
\begin{eqnarray}
E\left( \mathbf{M}\right) &=&\frac{A}{2\left\vert M_{s}\right\vert ^{2}}%
\int_{\Omega }\left\vert \nabla \mathbf{M}\right\vert ^{2}dr+\frac{K}{%
2\left\vert M_{s}\right\vert ^{2}}\int_{\Omega }\Phi \left( \mathbf{M}%
\right) dr-\int_{\Omega }\mathbf{H}_{\text{ext}}\cdot \mathbf{M}\,dr  \notag
\\
&&+\frac{1}{2}\int_{\mathbb{R}^{3}}\int_{{{\mathbb{R}^{3}}}}\frac{\nabla
\cdot \mathbf{M}\left( \mathbf{r}\right) \nabla \cdot \mathbf{M}\left( 
\mathbf{r^{\prime }}\right) }{\left\vert \mathbf{r-r^{\prime }}\right\vert }%
\,dr\,dr^{\prime }+\frac{M_{s}^{2}}{2K}\int_{\Omega }\left\vert \mathbf{H}_{%
\text{ext}}\right\vert ^{2}dr.  \label{fullenergy}
\end{eqnarray}%
Here, $\Omega \subset \mathbb{R}^{3}$ is the domain occupied by the
ferromagnetic material, $\mathbf{M}:\mathbb{R}^{3}\rightarrow \mathbb{R}^{3}$
is the magnetization vector that satisfies $\left\vert \mathbf{M}\right\vert
=M_{s}$ in $\Omega $ and $\mathbf{M}=0$ outside $\Omega $, the positive
constants $M_{s}$, $A$ and $K$ are the material parameters denoting the
saturation magnetization, exchange constant and the anisotropy constant,
respectively, $\mathbf{H}_{\text{ext}}$ is the applied external field, and $%
\Phi :\mathbb{R}^{3}\rightarrow \mathbb{R}$ is \ a nonnegative potential
which vanishes at finitely many points. {The divergence of} $\mathbf{M}$ in
the double integral is understood in the distributional sense. The five
terms in $\left( \ref{fullenergy}\right) $ represent the exchange energy,
the anisotropy energy, the Zeeman energy, the stray-field energy and an
inessential constant term added for convenience.

In the case of extended monocrystalline thin films with an in-plane easy
axis we have $\Omega =\mathbb{R}^{2}\times \left( 0,d\right) $. Without loss
of generality, we shall assume that the easy axis is in the $\mathbf{e}_{2}$
direction. Here $\mathbf{e}_{i}$ is the unit vector in the $i$-th coordinate
direction. For moderately soft thin films, a \textit{reduced thin} \textit{\
film energy }has been derived \cite{DKMO1,DKMO2,MO}, {providing a
significant simplification to the considered variational problem}. For a
better understanding of the parameter regime, we introduce the following
quantities%
\begin{equation*}
l=\left( \frac{A}{4\pi M_{s}^{2}}\right) ^{\frac{1}{2}},\text{ \ }L=\left( 
\frac{A}{K}\right) ^{\frac{1}{2}},\text{ }Q=\left( \frac{l}{L}\right) ^{2},
\end{equation*}%
representing the exchange length, the Bloch wall thickness and the material
quality factor, respectively. For \textit{ultra-thin }and \textit{soft }%
film, we have $d\lesssim l\lesssim L$, balanced as $Ld\sim l^{2}$. We can
then introduce a dimensionless parameter%
\begin{equation*}
\nu =\frac{4\pi M_{s}^{2}d}{KL}=\frac{Ld}{l^{2}}=\frac{d}{l\sqrt{Q}},
\end{equation*}%
{which measures the relative strength of the magnetostatic interaction.} For
the reduced thin film energy, we can write, {{after an appropriate
non-dimensionalization \cite{CM}:}} 
\begin{equation*}
E\left( \mathbf{m}\right) =\frac{1}{2}\int_{\mathbb{R}^{2}}\left( \left\vert
\nabla \mathbf{m}\right\vert ^{2}+\left( \mathbf{m\cdot e}_{1}-h\right)
^{2}\right) dr+\frac{\nu }{8 {\pi}}\int_{\mathbb{R}^{2}}\int_{\mathbb{R}^{2}}%
\frac{ \nabla \cdot \mathbf{m}\left( \mathbf{r}\right) \nabla \cdot \mathbf{m%
} \left( \mathbf{r^{\prime }}\right) }{\left\vert \mathbf{r-r^{\prime }}
\right\vert }\,dr\,dr^{\prime },
\end{equation*}%
where $\mathbf{m}:\mathbb{R}^{2}\rightarrow \mathbb{S}^{1}$ is the unit
magnetization vector in the film plane {and $h$ is the dimensionless applied
magnetic field}.

To study one-dimensional N\'{e}el wall profiles, we assume {further} that $%
\mathbf{m}$ depends only on {{$x = \mathbf{e}_{1} \cdot \mathbf{r}$}}.
Introducing the variable $\theta =\theta \left( x\right)$ that represents
the angle between $\mathbf{m}$ and the easy axis $\mathbf{e}_{2}$ in the
counter-clockwise direction, {we have} 
\begin{equation*}
\mathbf{m}\left( x\right) =\left( -\sin \theta \left( x\right) ,\cos \theta
\left( x\right) \right),
\end{equation*}%
for every $x\in \mathbb{R}$. One can rewrite the energy of such a
magnetization {configuration} per unit length of the wall in terms of $%
\theta $ as 
\begin{equation}
  E\left( \theta \right) =\frac{1}{2}\int_{\mathbb{R}}\left\{ \left(
        {d \theta \over dx} \right)^{2} +\left( \sin \theta -h\right)
    ^{2}+\frac{\nu }{2} 
    (\sin \theta - h)\left( -\frac{d^{2}}{dx^{2}}\right) ^{1/2} (\sin \theta -
    h)\right\} dx.  \label{energytheta}
\end{equation}%
Here $\left( -\frac{d^{2}}{dx^{2}}\right) ^{1/2}$ represents the linear
operator whose Fourier symbol is $\left\vert k\right\vert$ {{and can be
understood as a bounded linear map from $H^1(\mathbb{R})$, modulo additive
constants, to $L^2(\mathbb{R})$.}} Since two distinct global minima of the
energy in \eqref{energytheta} exist only if $\left\vert h\right\vert <1,$ we
shall always assume that $0\leq h<1$ in most of the paper.

\bigskip Let $\eta _{h}\in C^{\infty}( \mathbb{R}, [ 0,\pi])$ be a fixed
nonincreasing function with 
\begin{equation*}
\eta _{h}\left( x\right) =\left\{ 
\begin{array}{cl}
{\theta _{h}} & \text{if }x>1, \\ 
{\pi -\theta _{h}} & \text{if }x<-1,%
\end{array}%
\right. \qquad {\theta _{h}:=\arcsin h,}
\end{equation*}%
and consider an admissible class 
\begin{equation*}
\mathcal{A}:=\left\{ \theta \in H_{\text{loc}}^{1}\left( \mathbb{R}\right)
:\theta -\eta _{h}\in H^{1}\left( \mathbb{R}\right) \right\} .
\end{equation*}%
Note that the definition of $\mathcal{A}$ is independent of the choice of $%
\eta _{h}$. The following result was obtained in \cite{CM} addressing the
uniqueness, strict monotonicity, symmetry properties and decay of
one-dimensional N\'{e}el walls.

\begin{theorem}[\protect\cite{CM}]
  \label{cm}For every $\nu >0$ and every $h\in \left[ 0,1\right)$,
  there exists a minimizer of $E\left( \theta \right) $ in
  $\mathcal{A}$, which is unique (up to translations), strictly
  decreasing with the range equal to $%
  \left( \theta _{h},\pi -\theta _{h}\right) $
  and is smooth. Moreover, if $%
  \theta $
  \ is a minimizer satisfying
  $\theta \left( 0\right) =\frac{\pi }{2} , $ then
  $\theta \left( x\right) =\pi -\theta \left( -x\right)$, and there
  exists a constant $c>0$ such that
  $\lim_{x\rightarrow \infty }x^{2}\left( \theta \left( x\right)
    -\theta _{h}\right) =c$.
\end{theorem}

The Euler-Lagrange equation associated with the functional in $\left( \ref%
{energytheta}\right) $ is given by 
\begin{equation}
 - {d^2 \theta \over dx^2}+ \left( \sin \theta -h\right) {\cos \theta}
  +\frac{\nu }{2}
  \cos \theta \left( -\frac{d^{2}}{dx^{2}}\right) ^{1/2}\sin \theta =0,
\label{thetaequation}
\end{equation}%
with the boundary conditions at infinity%
\begin{equation}
\lim_{x\rightarrow +\infty }\theta \left( x\right) =\theta _{h},\text{ \ \ }
\lim_{x\rightarrow -\infty }\theta \left( x\right) =\pi -\theta _{h}.
\label{bc}
\end{equation}%
Our {main result is} the following uniqueness theoreom.

\begin{theorem}
\label{our}For every $\nu >0$ and every $h\in {\left[ 0,1\right) }$, there
exists a unique $\ $(up to translations) {non-increasing} smooth solution of
\ $\left( \ref{thetaequation}\right) $ which satisfies {the conditions at
infinity in} $\left( \ref{bc}\right) $ and has bounded energy.
\end{theorem}

\noindent Thus, the only possible monotone N\'eel wall profile is that of
the minimizer of the energy in \eqref{energytheta}, whose existence and
uniqueness was established in Theorem \ref{cm}. This confirms the
long-standing physical intuition that the N\'eel wall profiles observed in
ultrathin uniaxial ferromagnetic films minimize the one-dimensional
micromagnetic energy among all such profiles.

We also obtain the following estimates for the general one-dimensional
domain wall profiles. Here, by a one-dimensional domain wall
  profile we mean a smooth solution of \eqref{thetaequation}
  connecting zeroes of $\sin \theta -h$ at $x = \pm \infty$. From the
estimates in \cite{CMO} or \cite[section 5]{CM}, we know that any
solution $\theta$ of $\left( \ref{thetaequation}\right) $ with bounded
energy is smooth, and it is easy to see that any solution of
  \eqref{thetaequation} with bounded energy should approach a zero of
  $\sin \theta -h$ at infinity. We note that the obtained
estimates also apply to winding domain walls and, in particular, to
$360^{\circ }$ domain walls studied in \cite{mo:jap08,KKM}.

\begin{theorem}
  \label{uniestimate} There exist $C_{i}>0$, $i=1,2,\ldots $, such
  that for any solution $\theta$ of
  $\left( \ref{thetaequation}\right) $ with $%
  E\left( \theta \right) <\infty $ we have
\begin{equation*}
\sup_{x\in \mathbb{R}}\left( \left\vert \frac{d^{i}\theta }{dx^{i}}%
\right\vert \right) \leq C_{i},
\end{equation*}%
where $C_{i}=C_{i}\left( \nu ,h,E\left( \theta \right) \right) $. Moreover,
all the derivatives of $\theta $ vanish at infinity.
\end{theorem}

The main idea to prove the uniqueness result is as follows. Given any two
monotone solutions $\theta _{1}$ and $\theta _{2}$ of \eqref{thetaequation}
satisfying \eqref{bc} and $\theta_1(0) = \theta_2(0) = {\frac{\pi }{2}}$,
consider a suitable curve ${\gamma}$ connecting $\theta _{1}$ and $\theta
_{2}$. The curve $\gamma$ is chosen in such a way that any $\theta ^{{t}}\in
\gamma$ satisfies $\sin \theta ^{{t}}=t\sin \theta _{1}+\left( 1-t\right)
\sin \theta _{2}$ {for some $t\in \lbrack 0,1]$.} We {{then show that if $%
f\left( t\right) :=E\left( \theta ^{t}\right) $, then $f\in C^{2}([0,1])$
and $f^{\prime \prime }(t)>0$ for any }}$t\in \left[ 0,1\right] ${{, which
implies {strict} convexity of $f$.}} At the same time, since $\theta _{i}$
are solutions of $\left( \ref{thetaequation}\right) ,$ we must have$\left.
f^{\prime }\left( t\right) \right\vert _{t=0,1}=0,$ which is impossible. {A
similar} argument, {\ utilizing a hidden convexity of the considered energy
functional,} was used {\ recently} in \cite{GM} {to prove uniqueness of
solutions for a very different variational problem.}

The uniform bound theorem relies on the uniform estimate on the
nonlocal term in (\ref{thetaequation}). To obtain the estimate on the
nonlocal term, we used local smoothness of the solutions, together
  with the integral respresentation of the non-local term and
  energy-type estimates for the first derivatives. Decay property of
derivatives of solution at infinity follows directly once we get those
uniform derivative bounds.

\section{Uniqueness of {the} critical point}

Assume {{\ that $\theta _{1}\not\equiv \theta _{2}$}} are two non-increasing
solutions of $\left( \ref{thetaequation}\right) $ satisfying $\left( \ref{bc}%
\right) $ and $E(\theta _{i})<\infty $. By a suitable translation we can
ensure that $\theta _{i}\left( 0\right) =\frac{\pi }{2}$. Let now 
\begin{equation}
\theta ^{t}\left( x\right) {:=}\left\{ 
\begin{array}{cc}
\arcsin \left( t\sin \theta _{1}+\left( 1-t\right) \sin \theta _{2}\right) 
& x\geq 0, \\ 
\pi -\arcsin \left( t\sin \theta _{1}+\left( 1-t\right) \sin \theta
_{2}\right)  & x<0.%
\end{array}%
\right.   \label{thetat}
\end{equation}%
From the arguments of \cite{CM}, we know that $\theta _{i}$ are smooth
and $d \theta _{i} / dx <0$ on $\mathbb{R}$. We first prove the
following Lemma regarding differentiability of $\theta ^{t}$. We note
that the latter is not obvious a priori, since the definition of
$\theta ^{t}$ involves the arcsine function, which is \emph{not}
differentiable when its argument equals $\frac{\pi }{2}$. This could
potentially create problems near $x=0$. In fact, the conclusions
  of this section would clearly be incorrect, if there were multiple points  at which either 
  $\theta_1$ or $\theta_2$ equals $\frac{\pi}{2}$. Indeed, uniqueness of  solutions of
  \eqref{thetaequation} and \eqref{bc} with finite energy is false in view of the
  translational symmetry of the problem.  Therefore, the somewhat
  delicate estimates near $x = 0$ in the lemmas below are not merely
  technical, they are what enables the intuitive arguments of
  \cite{CMO,CM} to be used to establish uniqueness of the solutions
  that are translated so as to equal $\frac{\pi}{2}$ at $x = 0$.

  \medskip In the following, the subscripts $x$ and $t$ denote the
  partial derivatives with respect to the corresponding variables.

\begin{lemma}
  \label{derivativebound} For any $t\in \left[ 0,1\right]$, the
    function $\theta^{t}(x)$ is continuously differentiable with
  respect to $x\in \mathbb{R}$.  For any $x\in \mathbb{R}$,
  $\theta _{x}^{t} (x)$ is twice continuously differentiable with
  respect to $t$ on $\left[ 0,1%
  \right] $,
  with the understanding of one-sided derivatives at the boundary.
  All derivatives $\theta _{x}^{t} (x)$,
  $\theta _{xt}^{t}(x)$ and $\theta _{xtt}^{t}(x)$ are
  continuous functions of $x$ and $t$ separately on
  $\mathbb{R}\times \left[ 0,1\right]$. Moreover, there exists a
  constant $K > 0$ depending only on $\theta _{1}$ and
  $\theta _{2}$ such that for all $x \in \mathbb R$
\begin{eqnarray*}
 \max \left\{ \left\vert \theta _{x}^{t}\left( x\right)
  \right\vert  , \left\vert \theta _{xt}^{t}\left( x\right)
  \right\vert , \left\vert \theta _{xtt}^{t}\left( x\right)
  \right\vert   \right\} 
  \leq   K \left(\left\vert
  {d \theta _{1} \over dx} \left( x\right) \right\vert +
  \left\vert {d \theta _{2} \over dx} \left(
  x\right) \right\vert\right) \text{ \ for all\ }t\in \left[ 0,1\right] .
\end{eqnarray*}
\end{lemma}

\begin{proof}
\emph{\ }By our assumption, when $x\neq 0$ we have 
\begin{equation*}
0<t\sin \theta _{1}+\left( 1-t\right) \sin \theta _{2}<1
\end{equation*}%
for any $t\in \left[ 0,1\right] $.\emph{\ }Since arcsin($u$) is
differentiable for $u<1,$ chain rule applies when taking derivative of
$\theta ^{t}(x)$ with respect to $x$ at $x\neq 0$ for any
$t\in \left[ 0,1\right]$. From the assumption on $\theta _{i}$ and the
definition of $\theta ^{t}$, we have
\begin{equation*}
  \cos \theta ^{t}(x) =\text{sgn}\left( x\right)
  \sqrt{1-\sin ^{2}\theta^{t}(x) }.
\end{equation*}
Direct calculation then gives
\begin{eqnarray}
  \sin \theta ^{t} 
  &=&t\sin \theta _{1}+\left( 1-t\right) \sin \theta
      _{2}, \notag \\
  (\sin \theta ^{t})_{x} 
  &=&t\theta _{1x}\cos \theta _{1}+\left( 1-t\right)
      \theta _{2x}\cos \theta _{2}, \notag \\
  \theta _{x}^{t} 
  & = & \frac{ t\theta _{1x}\cos \theta
        _{1}+\left( 1-t\right) \theta _{2x}\cos \theta _{2} 
        }{\cos \theta ^{t}}, \qquad x\neq 0.  \label{thetatdx} 
\end{eqnarray}
Observe that when $x\neq 0$, the function $\frac{1}{\cos \theta ^{t}}$
is differentiable with respect to $t$ for any
$t\in \left[ 0,1\right]$. Differentiating
$\left( \ref{thetatdx}\right) $ with respect to $t$, we get for
$x\neq 0$
\begin{eqnarray}
  \theta _{xt}^{t}  
  &=&\frac{\theta _{1x}\cos \theta
      _{1}-\theta _{2x}\cos \theta _{2} }{\cos \theta ^{t}}
\label{thetatdxdt} \\
&&+\frac{\sin \theta ^{t}\left( \sin \theta _{1}-\sin \theta _{2}\right)
\left( t\theta _{1x}\cos \theta _{1}+\left( 1-t\right) \theta _{2x}\cos
\theta _{2}\right) }{\cos ^{3}\theta ^{t}}\text{ \ \ \ }  \notag
\end{eqnarray}
and
\begin{eqnarray}
  \theta _{xtt}^{t}
  &=&3\, \frac{\left( \sin \theta _{1}-\sin
      \theta _{2}\right) ^{2}\sin ^{2}\theta ^{t}\left( t\theta _{1x}\cos \theta
      _{1}+\left( 1-t\right) \theta _{2x}\cos \theta _{2}\right) }{\cos ^{5}\theta
      ^{t}}  \notag \\
  &&+2 \, \frac{\sin \theta ^{t}\left( \sin \theta _{1}-\sin \theta _{2}\right)
     \left( \theta _{1x}\cos \theta _{1}-\theta _{2x}\cos \theta
     _{2}\right) }{\cos ^{3}\theta ^{t}}  \notag \\ 
  &&+\frac{\left( \sin \theta _{1}-\sin \theta _{2}\right) ^{2}\left( t\theta
     _{1x}\cos \theta _{1}+\left( 1-t\right) \theta _{2x}\cos \theta _{2}\right) 
     }{\cos ^{3}\theta ^{t}}\text{\ }.  \label{thetatdxdtt}
\end{eqnarray}

From $\left( \ref{thetatdx}\right) ,$ $\left( %
  \ref{thetatdxdt}\right) $
and $\left( \ref{thetatdxdtt}\right) ,$ continuity of
$\theta _{x}^{t}$, $\theta _{xt}^{t}$ and $\theta _{xtt}^{t}$ with
respect to $x$ for all $x\neq 0$ follows. For $x=0,$\ we calculate the
derivatives of $\theta ^{t}$\ via the definition as follows. By
assumption, we have $0<\theta ^{t}\left( x\right) <\frac{\pi }{2}$
when $x>0$ and ${\frac{\pi }{2}}<$\
$\theta ^{t}\left( x\right) <\pi $\ when $x<0$. From this we obtain
\begin{eqnarray*}
  &&\lim_{x\rightarrow 0}\frac{\cos \theta ^{t}\left( x\right) }{x}  =
     \lim_{x\rightarrow 0}\frac{\left \vert \cos \theta^t (x)\right
     \vert}{|x|} = \lim_{x\rightarrow 0}\frac{\sqrt{1-\sin ^{2}\theta
     ^{t} (x)}}{|x|} \\ 
  &=&\lim_{x\rightarrow 0}\sqrt{t\left( \frac{1-\sin \theta
      _{1}(x)}{x^{2}}\right) +\left( 1-t\right) \left( \frac{1-\sin
      \theta _{2}(x)}{x^{2}}\right) }\times  \\ 
  &&\lim_{x\rightarrow 0}\sqrt{t\left( 1+\sin \theta _{1} (x)\right) +\left(
     1-t\right) \left( 1+\sin \theta _{2} (x)\right) } \\
  &=&\sqrt{t\theta _{1x}^{2}\left( 0\right) +\left( 1-t\right) \theta
      _{2x}^{2}\left( 0\right) },
\end{eqnarray*}
The last step in the limit above follows from applying L'Hospital's
rule in
\begin{eqnarray}
  \lim_{x\rightarrow 0}\frac{1-\sin \theta _{i}(x)}{x^{2}}
  & = & -
        \lim_{x\rightarrow 0} 
        \frac{\theta _{ix} (x) \cos \theta _{i} (x)}{2x} \notag \\
  & = & \lim_{x\rightarrow 0}\frac{\theta
        _{ix}^{2} (x) \sin \theta _{i} (x)-\theta _{ixx} (x) \cos \theta
        _{i} (x)}{2}=\frac{1}{2}%
        \theta _{ix}^{2}\left( 0\right) .  \label{sinlimit}
\end{eqnarray}

We calculate the derivative of $\ \theta ^{t}(x)$ with respect to $x$
at $x=0$ as follows:
\begin{eqnarray}
  \theta _{x}^{t}\left( 0\right)  
  &=&\lim_{x\rightarrow 0}\frac{\theta
      ^{t}\left( x\right) -\theta ^{t}\left( 0\right) }{x}  \notag \\
  &=&\lim_{x\rightarrow 0}\frac{\theta ^{t}\left( x\right) -\theta ^{t}\left(
      0\right) }{\sin \left( \theta ^{t}\left( x\right) -\theta ^{t}\left(
      0\right) \right) }\times \frac{\sin \left( \theta ^{t}\left( x\right)
      -\theta ^{t}\left( 0\right) \right) }{x}  \notag \\
  &=&\lim_{x\rightarrow 0}\frac{\sin \theta ^{t}\left( x\right) \cos \theta
      ^{t}\left( 0\right) -\sin \theta ^{t}\left( 0\right) \cos \theta ^{t}\left(
      x\right) }{x}  \notag \\
  &=&-\lim_{x\rightarrow 0}\frac{\cos \theta ^{t}\left( x\right) }{x}  
      = -\sqrt{t\theta _{1x}^{2}\left( 0\right) +\left( 1-t\right) \theta
      _{2x}^{2}\left( 0\right) }.  \label{thetatdxzero}
\end{eqnarray}%
Moreover, by $\left( \ref{sinlimit}\right) $ we have 
\begin{eqnarray}
  &&\lim_{x\rightarrow 0}\frac{\cos \theta _{i} (x) }{\cos \theta ^{t} (x)
     }  = \lim_{x\rightarrow 0}\sqrt{\left( \frac{1-\sin \theta _{i}
     (x) }{x^{2}}\right) 
     \left( 1+\sin \theta _{i} (x) \right) }\times   \notag \\
  &&\lim_{x\rightarrow 0}\frac{1}{\sqrt{t\left( \frac{1-\sin \theta
     _{1} (x) }{x^{2}%
     }\right) +\left( 1-t\right) \left( \frac{1-\sin \theta _{2} (x) }{x^{2}}\right) }}%
     \times   \notag \\
  &&\lim_{x\rightarrow 0}\frac{1}{\sqrt{t\left( 1+\sin \theta _{1} (x)
     \right)
     +\left( 1-t\right) \left( 1+\sin \theta _{2} (x)\right) }}  \notag \\
  &=&-\frac{\theta _{ix}\left( 0\right) }{\sqrt{t\theta _{1x}^{2}\left(
      0\right) +\left( 1-t\right) \theta _{2x}^{2}\left( 0\right) }}.
\label{cosratior}
\end{eqnarray}%
where in the last step we used the fact that $\theta _{ix}\left( 0\right) <0.
$ It then follows from $\left( \ref{cosratior}\right) $ that
\begin{eqnarray}
  &&\lim_{x\rightarrow 0}\theta _{x}^{t}\left( x\right)  
     = \lim_{x\rightarrow 0}\frac{\left( t\theta _{1x} (x) \cos \theta
     _{1} (x) +\left(
     1-t\right) \theta _{2x} (x) \cos \theta _{2} (x) \right) }{\cos\theta
     ^{t} (x)}  \notag \\
  &=&-\frac{t\theta _{1x}^{2}\left( 0\right) +\left( 1-t\right) \theta
      _{2x}^{2}\left( 0\right) }{\sqrt{t\theta _{1x}^{2}\left( 0\right) +\left(
      1-t\right) \theta _{2x}^{2}\left( 0\right) }}  = -\sqrt{t\theta
      _{1x}^{2}\left( 0\right) +\left( 1-t\right) \theta 
      _{2x}^{2}\left( 0\right) },  \label{thetatdxlimitzeror}
\end{eqnarray}%
Equations $\left( \ref{thetatdxzero}\right) $ and $\left( \ref%
  {thetatdxlimitzeror}\right) $\
\ imply that $\theta _{x}^{t}(x)$ is continuous at $x=0$ for any
$t\in \left[ 0,1\right] $. Continuity of $\theta^{t}_x$ with respect
to $t$ is obvious from $\left(\ref{thetatdx}\right) $ and
$\left( \ref{thetatdxzero}\right) $.

Next we evaluate $\theta _{xt}^{t}$ at $x=0$. Recall that $%
\theta _{ix}\left( 0\right) <0$ and differentiate $\left( \ref{thetatdxzero}%
\right) $\ with respect to $t$. We get 
\begin{equation}
  \theta _{xt}^{t}\left( 0\right) =\frac{\theta _{2x}^{2}\left(
      0\right) -\theta _{1x}^{2}\left( 0\right) }{2\sqrt{t\theta _{1x}^{2}\left(
        0\right) +\left( 1-t\right) \theta _{2x}^{2}\left( 0\right) }}.
\label{thetatdxdtzero}
\end{equation}%
On the other hand, $\left( \ref{sinlimit}\right) $ yields 
\begin{eqnarray}
  &&\lim_{x\rightarrow 0}\frac{\sin \theta _{1} (x) -\sin \theta _{2}
     (x) }{1-\sin
     ^{2}\theta ^{t} (x)}  \notag \\
  &=&\lim \frac{\frac{\sin \theta _{1} (x) -1}{x^{2}}+\frac{1-\sin
      \theta _{2} (x) }{%
      x^{2}}}{t\frac{\left( 1-\sin \theta _{1} (x) \right) }{x^{2}}+\left( 1-t\right) 
      \frac{1-\sin \theta _{2} (x) }{x^{2}}}\times   \notag \\
  &&\lim_{x\rightarrow 0}\frac{1}{t\left( 1+\sin \theta _{1} (x) \right) +\left(
     1-t\right) \left( 1+\sin \theta _{2} (x) \right) }  \notag \\
  &=&\frac{-\frac{1}{2}\theta _{1x}^{2}\left( 0\right) +\frac{1}{2}\theta
      _{2x}^{2}\left( 0\right) }{t\theta _{1x}^{2}\left( 0\right) +\left(
      1-t\right) \theta _{2x}^{2}\left( 0\right) }  \label{sinratio}
\end{eqnarray}%
Using $\left( \ref{cosratior}\right) $ \ and
$\left( \ref{sinratio}\right) $%
, we evaluate the limit of $\theta _{xt}^{t}$ at $x=0$ as
\begin{eqnarray}
  &&\lim_{x\rightarrow 0} \theta _{xt}^{t}\left( x\right)  
     \notag \\
  &=&\lim_{x\rightarrow 0}\frac{\theta _{1x} (x) \cos \theta
      _{1} (x) -\theta
      _{2x} (x) \cos \theta _{2} (x)  }{\cos \theta ^{t} (x) }  \notag \\
  &&+\lim_{x\rightarrow 0}\frac{\sin \theta ^{t} (x) \left( \sin
     \theta _{1} (x) -\sin
     \theta _{2} (x) \right) \left( t\theta _{1x} (x) \cos \theta _{1}
     (x) +\left( 1-t\right)
     \theta _{2x} (x) \cos \theta _{2} (x) \right) }{\cos ^{3}\theta
     ^{t} (x) }
     \notag \\
  &=&\frac{-\theta _{1x}^{2}\left( 0\right) +\theta _{2x}^{2}\left( 0\right) }{%
      \sqrt{t\theta _{1x}^{2}\left( 0\right) +\left( 1-t\right) \theta
      _{2x}^{2}\left( 0\right) }}  \notag \\
  &&+\lim_{x\rightarrow 0}\frac{\sin \theta _{1} (x) -\sin \theta _{2}
     (x) }{1-\sin
     ^{2}\theta ^{t} (x) }\times \frac{t\theta _{1x} (x) \cos \theta
     _{1} (x) +\left( 1-t\right)
     \theta _{2x} (x) \cos \theta _{2} (x) }{\cos \theta ^{t} (x)}  \notag \\
  &=&\frac{-\theta _{1x}^{2}\left( 0\right) +\theta _{2x}^{2}\left( 0\right) }{%
      \sqrt{t\theta _{1x}^{2}\left( 0\right) +\left( 1-t\right) \theta
      _{2x}^{2}\left( 0\right) }}  \notag \\
  &&-\frac{-\frac{1}{2}\theta _{1x}^{2}\left( 0\right) +\frac{1}{2}\theta
     _{2x}^{2}\left( 0\right) }{t\theta _{1x}^{2}\left( 0\right) +\left(
     1-t\right) \theta _{2x}^{2}\left( 0\right) }\times \frac{t\theta
     _{1x}^{2}\left( 0\right) +\left( 1-t\right) \theta _{2x}^{2}\left( 0\right) 
     }{\sqrt{t\theta _{1x}^{2}\left( 0\right) +\left( 1-t\right) \theta
     _{2x}^{2}\left( 0\right) }}  \notag \\
  &=&\frac{1}{2}\frac{-\theta _{1x}^{2}\left( 0\right) +\theta _{2x}^{2}\left(
      0\right) }{\sqrt{t\theta _{1x}^{2}\left( 0\right) +\left( 1-t\right) \theta
      _{2x}^{2}\left( 0\right) }}.  \notag \\
  \label{thetatdxdtlimitzeror}
\end{eqnarray}%
We conclude from $\left( \ref{thetatdxdtzero}\right) $ and
$\left( \ref%
  {thetatdxdtlimitzeror}\right) $
that $\theta _{xt}^{t}(x)$ is continuous at $x=0$ and continuity of
$\theta _{xt}^{t}(x)$ with respect to $t$ follows from
$\left( \ref{thetatdxdt}\right) $ and $%
\left( \ref{thetatdxdtzero}\right) $.
Lastly, recall that $\theta _{ix}\left( 0\right) <0$ and differentiate
$%
\left( \ref{thetatdxdtzero}\right) $ with respect to $t$. This yields
\begin{equation} 
  \theta _{xtt}^{t}\left( 0\right)
  =\frac{1}{4}\frac{\left( \theta _{2x}^{2}\left( 0\right) -\theta
      _{1x}^{2}\left( 0\right) \right) ^{2}%
  }{\left( \sqrt{t\theta _{1x}^{2}\left( 0\right) +\left( 1-t\right)
        \theta _{2x}^{2}\left( 0\right) }\right)
    ^{3}}.  \label{thetatdxdttzero}
\end{equation}%

To derive continuity of $\theta _{xtt}^{t}(x)$ at $x=0$, we calculate
the limit of $\theta _{xtt}^{t}(x)$ at $x=0$. By $%
\left( \ref{cosratior}\right) $ and $%
\left( \ref{sinratio}\right) $,
\begin{eqnarray}
  &&\lim_{x\rightarrow 0}\theta _{xtt}^{t}\left( x\right)   \notag \\
  &=&\lim_{x\rightarrow 0}\left\{ 3 \, \frac{\left( \sin \theta _{1}
      (x) -\sin \theta 
      _{2} (x) \right) ^{2}\sin ^{2}\theta ^{t} (x) \left( t\theta
      _{1x} (x) \cos \theta
      _{1} (x) +\left( 1-t\right) \theta _{2x} (x) \cos \theta _{2}
      (x) \right)
      }{\cos ^{5}\theta ^{t} (x)}\right.   \notag \\ 
  &&+2 \, \frac{\sin \theta ^{t} (x) \left( \sin \theta _{1} (x) -\sin
     \theta _{2} (x) \right)
     \left( \theta _{1x} (x) \cos \theta _{1} (x) -\theta _{2x} (x) \cos \theta
     _{2} (x) \right) } {\cos ^{3}\theta ^{t} (x)}  \notag \\
  &&+\left. \frac{\left( \sin \theta _{1} (x) -\sin \theta _{2} (x) \right) ^{2}\left(
     t\theta _{1x} (x) \cos \theta _{1} (x) +\left( 1-t\right) \theta
     _{2x} (x) \cos \theta
     _{2} (x) \right) }{\cos ^{3}\theta ^{t} (x) } \right\}   \notag \\
  &=&-3\left( \frac{-\frac{1}{2}\theta _{1x}^{2}\left( 0\right) +\frac{1}{2}%
      \theta _{2x}^{2}\left( 0\right) }{t\theta _{1x}^{2}\left( 0\right) +\left(
      1-t\right) \theta _{2x}^{2}\left( 0\right) }\right) ^{2}\times \sqrt{t\theta
      _{1x}^{2}\left( 0\right) +\left( 1-t\right) \theta _{2x}^{2}\left( 0\right) }
      \notag \\
  &&+2 \, \frac{-\frac{1}{2}\theta _{1x}^{2}\left( 0\right) +\frac{1}{2}\theta
     _{2x}^{2}\left( 0\right) }{t\theta _{1x}^{2}\left( 0\right) +\left(
     1-t\right) \theta _{2x}^{2}\left( 0\right) }\frac{-\theta _{1x}^{2}\left(
     0\right) +\theta _{2x}^{2}\left( 0\right) }{\sqrt{t\theta _{1x}^{2}\left(
     0\right) +\left( 1-t\right) \theta _{2x}^{2}\left( 0\right) }}  \notag \\
  &=&\frac{1}{4}\frac{\left( \theta _{2x}^{2}\left( 0\right) -\theta
      _{1x}^{2}\left( 0\right) \right) ^{2}}{\left( \sqrt{t\theta _{1x}^{2}\left(
      0\right) +\left( 1-t\right) \theta _{2x}^{2}\left( 0\right)
      }\right) ^{3}}. \notag \\
\label{thetatdxdttlimitzeror}
\end{eqnarray}%
Continuity of $\theta _{xtt}^{t}(x)$ at $x=0$ follows from $%
\left( \ref{thetatdxdttzero}\right) $ and $\left( \ref%
  {thetatdxdttlimitzeror}\right) $.
Continuity of $\theta _{xtt}^{t}(x)$ with respect to $t$ variable
follows from $\left( \ref%
  {thetatdxdtt}\right) $ and $\left( \ref{thetatdxdttzero}\right) $.

Finally, we derive the bounds on those derivatives. When $x=0,$
it follows directly from $\left( \ref{thetatdxzero}\right) ,$
$\left( \ref{thetatdxdtzero}%
\right) $ and $\left( \ref{thetatdxdttzero}\right) $ that
\begin{equation}
\left\vert \theta _{x}^{t}\left( 0\right) \right\vert \leq \left\vert
\theta _{1x}\left( 0\right) \right\vert +\left\vert \theta _{2x}\left(
0\right) \right\vert , \label{thetaxzero}
\end{equation}
\begin{equation}
  \left\vert \theta _{xt}^{t}\left( 0\right) \right\vert \leq
  \frac{1}{2}\max \left( \left\vert \frac{\theta _{1x}\left( 0\right)
      }{\theta  
        _{2x}\left( 0\right) }\right\vert ,\left\vert \frac{\theta _{2x}\left(
          0\right) }{\theta _{1x}\left( 0\right) }\right\vert \right) \left(
    \left\vert \theta _{1x}\left( 0\right) \right\vert +\left\vert \theta
      _{2x}\left( 0\right) \right\vert \right) , \label{thetaxtzero}
\end{equation}
\begin{equation}
\left\vert \theta _{xtt}^{t}\left( 0\right) \right\vert \leq \frac{1}{4}\max \left( \left\vert \frac{\theta _{1x}\left( 0\right) }{\theta
_{2x}\left( 0\right) }\right\vert ^3,\left\vert \frac{\theta _{2x}\left(
0\right) }{\theta _{1x}\left( 0\right) }\right\vert ^3\right) \left(
\left\vert \theta _{1x}\left( 0\right) \right\vert +\left\vert \theta
_{2x}\left( 0\right) \right\vert \right) . \label{thetaxttzero}
\end{equation}
To obtain the bound for $x\neq 0,${\ we write} 
\begin{eqnarray*}
\sqrt{1-\left( t\sin \theta _{1}+\left( 1-t\right) \sin \theta _{2}\right)
^{2}} &\geq &\sqrt{t\left( 1-\sin \theta _{1}\right) +\left( 1-t\right)
\left( 1-\sin \theta _{2}\right) } \\
&\geq &t\sqrt{1-\sin \theta _{1}}+\left( 1-t\right) \sqrt{1-\sin \theta _{2}}%
,
\end{eqnarray*}%
where the last inequality follows from the concavity of the
  function $F(s) = \sqrt{s}$. It then follows that for all
{$0\leq t\leq 1$ and $x\not=0$} we have%
\begin{eqnarray}
  &&\left\vert \frac{\left( t\theta _{1x}\cos \theta _{1}+(1-t)\theta
     _{2x}\cos \theta _{2}\right) }{\cos \theta ^{t}} \right\vert  
     \notag \\
  &\leq &\left\vert \theta _{1x}\right\vert \frac{\sqrt{{2}\left( 1-\sin
          \theta _{1}\right) }}{t\sqrt{1-\sin \theta _{1}}+\left( 1-t\right) \sqrt{%
          1-\sin \theta _{2}}}  \notag \\
  &&+\left\vert \theta _{2x}\right\vert \frac{\sqrt{2(1-\sin \theta _{2})}}{t%
     \sqrt{1-\sin \theta _{1}}+\left( 1-t\right) \sqrt{1-\sin \theta _{2}}} 
     \notag \\
  &\leq &\sqrt{2} \, N\left( x\right) \left( \left\vert \theta
          _{1x}\right\vert +\left\vert \theta _{2x}\right\vert \right)
          ,  \label{thetax} 
\end{eqnarray}%
where
\begin{equation*}
N\left( x\right) =\max \left( {{\frac{\sqrt{1-\sin \theta _{1}\left(
x\right) }}{\sqrt{1-\sin \theta _{2}\left( x\right) }},}}\text{ }\frac{\sqrt{%
1-\sin \theta _{2}\left( x\right) }}{\sqrt{1-\sin \theta _{1}\left( x\right) 
}}\right) .
\end{equation*}%
Similarly, we have 
\begin{eqnarray}
  &&\left\vert \frac{\left( \theta _{1x}\cos \theta _{1}-\theta _{2x}\cos
     \theta _{2}\right) }{\cos \theta ^{t}} \right\vert   \notag \\
  &\leq &\left\vert \theta _{1x}\right\vert \frac{\sqrt{{2}\left( 1-\sin
          \theta _{1}\right) }}{t\sqrt{1-\sin \theta _{1}}+\left( 1-t\right) \sqrt{%
          1-\sin \theta _{2}}}  \notag \\
  &&+\left\vert \theta _{2x}\right\vert \frac{\sqrt{2(1-\sin \theta _{2})}}{t%
     \sqrt{1-\sin \theta _{1}}+\left( 1-t\right) \sqrt{1-\sin \theta _{2}}} 
     \notag \\
  &\leq &\sqrt{2} \, N\left( x\right) \left( \left\vert \theta _{1x}\right\vert
          +\left\vert \theta _{2x}\right\vert \right) ,  \label{thetaxt1}
\end{eqnarray}%
\begin{eqnarray}
  &&\left\vert \frac{\left( \sin \theta _{1}-\sin \theta _{2}\right) \left(
     t\theta _{1x}\cos \theta _{1}+\left( 1-t\right) \theta _{2x}\cos \theta
     _{2}\right) }{\cos ^{3}\theta ^{t}}\right\vert   \notag \\
  &\leq &\left\vert \theta _{1x}\right\vert \frac{\left( \sqrt{1-\sin \theta
          _{1}}\right) ^{3}+\sqrt{1-\sin \theta _{1}}\left( 1-\sin \theta _{2}\right) 
          }{\left( t\sqrt{1-\sin \theta _{1}}+\left( 1-t\right) \sqrt{1-\sin \theta
          _{2}}\right) ^{3}}  \notag \\
  &&+\left\vert \theta _{2x}\right\vert \frac{\sqrt{1-\sin \theta _{2}}\left(
     1-\sin \theta _{1}\right) +\left( \sqrt{1-\sin \theta _{2}}\right) ^{3}}{%
     \left( t\sqrt{1-\sin \theta _{1}}+\left( 1-t\right) \sqrt{1-\sin \theta _{2}}%
     \right) ^{3}}  \notag \\
  &\leq &\left( N ^{3}\left( x\right) + N \left( x\right) \right) \left(
          \left\vert \theta _{1x}\right\vert +\left\vert \theta _{2x}\right\vert
          \right) ,  \label{thetaxt2}
\end{eqnarray}%
\begin{eqnarray}
  &&\left\vert \frac{\left( \sin \theta _{1}-\sin \theta _{2}\right) \left(
     \theta _{1x}\cos \theta _{1}-\theta _{2x}\cos \theta _{2}\right)
     }{\cos^3\theta^t} \right\vert   \notag \\
  &\leq &\left\vert \theta _{1x}\right\vert \frac{\left( \sqrt{1-\sin \theta
          _{1}}\right) ^{3}+\sqrt{1-\sin \theta _{1}}\left( 1-\sin \theta _{2}\right) 
          }{\left( t\sqrt{1-\sin \theta _{1}}+\left( 1-t\right) \sqrt{1-\sin \theta
          _{2}}\right) ^{3}}  \notag \\
  &&+\left\vert \theta _{2x}\right\vert \frac{\sqrt{1-\sin \theta _{2}}\left(
     1-\sin \theta _{1}\right) +\left( \sqrt{1-\sin \theta _{2}}\right) ^{3}}{%
     \left( t\sqrt{1-\sin \theta _{1}}+\left( 1-t\right) \sqrt{1-\sin \theta _{2}}%
     \right) ^{3}}  \notag \\
  &\leq &\left( N^3 \left( x\right)+ N \left( x\right) \right)
          \left( \left\vert \theta _{1x}\right\vert +\left\vert \theta
          _{2x}\right\vert \right) ,  \label{thetaxtt1}
\end{eqnarray}%
\begin{eqnarray}
  &&\frac{\left( \sin \theta _{1}-\sin \theta _{2}\right) ^{2}\left( t\theta
     _{1x}\cos \theta _{1}+\left( 1-t\right) \theta _{2x}\cos \theta _{2}\right) 
     } {\cos^5\theta^t}  \notag \\
  &\leq &\left\vert \theta _{1x}\right\vert \frac{\left( \sqrt{1-\sin \theta
          _{1}}\right) ^{5}+\sqrt{1-\sin \theta _{1}}\left( 1-\sin \theta _{2}\right)
          ^{2}}{\left( t\sqrt{1-\sin \theta _{1}}+\left( 1-t\right) \sqrt{1-\sin
          \theta _{2}}\right) ^{5}}  \notag \\
  &&+\left\vert \theta _{2x}\right\vert \frac{\sqrt{1-\sin \theta _{2}}\left(
     1-\sin \theta _{1}\right) +\left( \sqrt{1-\sin \theta _{2}}\right) ^{3}}{%
     \left( t\sqrt{1-\sin \theta _{1}}+\left( 1-t\right) \sqrt{1-\sin \theta _{2}}%
     \right) ^{3}}  \notag \\
  &\leq &\left( N^5 \left( x\right)+ N\left( x\right) \right)
          \left( \left\vert \theta _{1x}\right\vert +\left\vert \theta
          _{2x}\right\vert \right) .  \label{thetaxtt2}
\end{eqnarray}%
By $\left( \ref{sinlimit}\right)$, there exists $\delta _{0}>0$ such
that for $i=1,2$ we have
\begin{equation*}
\frac{1}{4}\theta _{ix}^{2}\left( 0\right) <\frac{1-\sin \theta _{i}\left(
x\right) }{x^{2}}<\frac{3}{4}\theta _{ix}^{2}\left( 0\right) \text{ \ when }%
\left\vert x\right\vert <\delta _{0}.
\end{equation*}%
Therefore, for $\left\vert x\right\vert <\delta _{0}$ we obtain
\begin{equation}
{\max }\left\{ {{\frac{\sqrt{1-\sin \theta _{1}\left( x\right) }}{\sqrt{%
          1-\sin \theta _{2}\left( x\right) }},}}\text{ }\frac{\sqrt{1-\sin \theta
      _{2}\left( x\right) }}{\sqrt{1-\sin \theta _{1}\left( x\right) }}\right\}
<3\max \left\{ \frac{\theta _{1x}\left( 0\right) }{\theta _{2x}\left(
      0\right) },\frac{\theta _{2x}\left( 0\right) }{\theta _{1x}\left( 0\right) }%
\right\},  \label{ratioboundzero}
\end{equation}%
and for $\left\vert x\right\vert \geq \delta _{0},$ we get
\begin{eqnarray}
  &&{\max }\left\{ {{\frac{\sqrt{1-\sin \theta _{1}\left( x\right) }}{\sqrt{%
     1-\sin \theta _{2}\left( x\right) }},}}\text{ }\frac{\sqrt{1-\sin \theta
     _{2}\left( x\right) }}{\sqrt{1-\sin \theta _{1}\left( x\right) }}\right\}  
     \notag \\
  &\leq &\max \left\{ \frac{1-h}{1-\sin \theta _{1}\left( \delta _{0}\right) },%
          \frac{1-h}{1-\sin \theta _{1}\left( -\delta _{0}\right) },\frac{1-h}{1-\sin
          \theta _{2}\left( \delta _{0}\right) },\frac{1-h}{1-\sin \theta _{2}\left(
          -\delta _{0}\right) }\right\} .  \notag \\
  &=&L_{\delta _{0}}.  \notag \\
  \label{ratioboundaway}
\end{eqnarray}%

Let
\begin{equation*}
  M=\max \left\{1, ,L_{\delta _{0}}, 
    \left(\frac{\theta_{1x}(0)}{\theta_{2x}(0)}\right)^3,
    \left(\frac{\theta_{2x}(0)}{\theta_{1x}(0)}\right)^3,  
    3 \, \frac{\theta _{1x}\left( 0\right) }{\theta _{2x}\left( 
        0\right) }, 3 \, \frac{\theta _{2x}\left( 0\right) }{\theta
      _{1x}\left( 0\right) } \right\} .
\end{equation*}
Then equations $\left( \ref{thetatdx}\right) ,$
$\left( \ref{thetatdxdt}\right) ,$
$\left( \ref{thetatdxdtt}\right) ,\left( \ref%
  {thetaxzero}\right) -$
$\left( \ref{thetaxttzero}\right) ,$ $\left( \ref%
  {thetax}\right) -\left( \ref{thetaxtt2}\right) ,\left(
  \ref{ratioboundzero}%
\right) $ and $\left( \ref{ratioboundaway}\right) $ imply
\begin{eqnarray*}
  \left\vert \theta _{x}^{t}\left( x\right) \right\vert  
  &\leq &M (\left\vert
          \theta _{1x}\left( x\right) \right\vert +\left\vert \theta _{2x}\left(
          x\right) \right\vert) , \\
  \left\vert \theta _{xt}^{t}\left( x\right) \right\vert  
  &\leq &\left(
          M^{3}+M\right) \left( \left\vert \theta _{1x}\left( x\right) \right\vert
          +\left\vert \theta _{2x}\left( x\right) \right\vert \right) , \\
  \left\vert \theta _{xtt}^{t}\left( x\right) \right\vert  
  &\leq &\left(
          M^{5}+M\right) \left( \left\vert \theta _{1x}\left( x\right) \right\vert
          +\left\vert \theta _{2x}\left( x\right) \right\vert \right) ,
\end{eqnarray*}%
for all $x \in \mathbb R$. The conclusion then follows by taking
$K=M^{5}+M$.
\end{proof}

\bigskip Recall that
\begin{eqnarray}
  &&E\left( \theta ^{t}\right)   \notag \\
  &=&\frac{1}{2}\int_{\mathbb{R}}\left\{ \left\vert \theta _{x}^{t}\right\vert
      ^{2}+\left( \sin \theta ^{t}-h\right) ^{2}+\frac{\nu }{2}(\sin \theta
      ^{t}-h)\left( -\frac{d^{2}}{dx^{2}}\right) ^{1/2}(\sin \theta
      ^{t}-h)\right\} dx  \notag \\
  &=&\frac{1}{2}\int_{\mathbb{R}}\left( \left\vert \theta _{x}^{t}\right\vert
      ^{2}+\left( t\left( \sin \theta _{1}-h\right) +\left( 1-t\right) \left( \sin
      \theta _{2}-h\right) \right) ^{2}\right) dx  \notag \\
  &&+\frac{\nu }{4} t^{2}\, \int_{\mathbb{R}} \left( \sin \theta _{1}-h\right)
     \left( -\frac{d^{2}}{dx^{2}}\right) ^{1/2}(\sin \theta _{1}-h)dx\,  \notag \\
  &&+\frac{\nu }{2} t\left( 1-t\right) \, \int_{\mathbb{R}}\left( \sin \theta
     _{1}-h\right) \left( -\frac{d^{2}}{dx^{2}}\right) ^{1/2}(\sin \theta
     _{2}-h)dx\,  \notag \\
  &&+\frac{\nu }{4} \left( 1-t\right) ^{2}\, \int_{\mathbb{R}}\left( \sin \theta
     _{2}-h\right) \left( -\frac{d^{2}}{dx^{2}}\right) ^{1/2}(\sin \theta
     _{2}-h)dx\,  \notag \\
  &&  \label{Ethetat}
\end{eqnarray}
We shall write $f\left( t\right) =E\left( \theta ^{t}\right) $ for
shorthand. The following lemma is a direct corollary of Lemma $\ref%
{derivativebound}$.

\begin{lemma}
\label{Ederivative}We have $E\left( \theta ^{t}\right) <\infty $ for all $t
\in [0,1]$. Moreover, $f\left( t\right) $\ is twice continuously
differentiable, and $f^{\prime \prime }\left( t\right) > 0$ on $\left[ 0,1%
\right] $.\ 
\end{lemma}

\begin{proof}
By Lemma \ref{derivativebound} and $\left( \ref{Ethetat}\right) ,$ we have 
\begin{equation*}
E\left( \theta ^{t}\right) \leq K^{2}E\left( \theta _{1}\right)
+K^{2}E\left( \theta _{2}\right) .
\end{equation*}%
Therefore, $f(t)=E(\theta ^{t})$ is well defined. To ensure that
$f(t)$ is sufficiently regular, observe that from
$\left( \ref{Ethetat}\right) $ we can write
\begin{equation*}
E\left( \theta ^{t}\right) =\int_{\mathbb{R}}\frac{1}{2}\left\vert \theta
_{x}^{t}\right\vert ^{2}dx+P_{2}\left( t\right) ,
\end{equation*}%
where $P_{2}\left( t\right) $ is a quadratic polynomial in $t$ with bounded
coefficients depending on $\theta _{1,}$ $\theta _{2}$. The question of
differentiability of $f\left( t\right) $ thus reduces to that of 
\begin{equation*}
g\left( t\right) =\int_{\mathbb{R}}\frac{1}{2}\left\vert \theta
_{x}^{t}\right\vert ^{2}dx.
\end{equation*}%
By Lemma \ref{derivativebound}, for all $x \in \mathbb R$ we have
\begin{equation}
\left\vert \theta _{x}^{t}\left( x\right) \theta _{xt}^{t}\left( x\right)
\right\vert \leq 2K^{2}\left( \theta _{1x}^{2}\left( x\right) +\theta
_{2x}^{2}\left( x\right) \right) \text{ \ for all }t\in \left[
0,1\right], 
\label{dominate1}
\end{equation}%
and 
\begin{equation}
\left\vert \theta _{xt}^{t} (x) \right\vert ^{2}+\left\vert \theta
  _{x}^{t} (x) \theta
_{xtt}^{t} (x) \right\vert \leq 6K^{2}\left( \theta _{1x}^{2}\left( x\right)
+\theta _{2x}^{2}\left( x\right) \right) \text{\ for all \ }t\in \left[ 0,1%
\right].  \label{dominate2}
\end{equation}%
Since $E(\theta _{i})<\infty $ and
$\theta _{x}^{t}\left( x\right) ,\theta _{xt}^{t}\left( x\right) $ and
$\theta _{xtt}^{t}\left( x\right) $ are continuous in $t$ on
$\left[ 0,1\right] $ for each $x\in \mathbb{R}$, we conclude from
dominated convergence theorem and continuity of integral theorem that
for $t\in \left[ 0,1\right] $%
\begin{eqnarray*}
  g^{\prime }\left( t\right)  
  =\int_{\mathbb{R}}\theta _{x}^{t}  \theta _{xt}^{t} \, dx,  \qquad
  g^{\prime \prime }\left( t\right)  
  =\int_{\mathbb{R}}\left( \left\vert
  \theta _{xt}^{t}\right\vert ^{2}+\theta _{x}^{t}\theta _{xtt}^{t}\right) dx,
\end{eqnarray*}%
and that $g^{\prime }\left( t\right) $ and
$g^{\prime \prime }\left( t\right) $ are both continuous on
$\left[ 0,1\right] $. A direct computation then yields
\begin{eqnarray*}
  &&\frac{d^{2}\left( E\left( \theta ^{t}\right) \right) }{dt^{2}} \\
  &=&\int_{\mathbb{R}}\left( \left\vert \theta _{xt}^{t}\right\vert
      ^{2}+\theta _{x}^{t}\theta _{xtt}^{t}\right) dx+\int_{\mathbb{R}}\left( \sin
      \theta _{1}-\sin \theta _{2}\right) ^{2} dx \\
  &&+\frac{\nu }{2}\int_{\mathbb{R}}\left( \sin \theta _{1}-h\right) \left( -%
     \frac{d^{2}}{dx^{2}}\right) ^{1/2}(\sin \theta _{1}-h) \, dx \\
  &&+\frac{\nu }{2}\int_{\mathbb{R}}\left( \sin \theta _{2}-h\right) \left( -%
     \frac{d^{2}}{dx^{2}}\right) ^{1/2}(\sin \theta _{2}-h) \, dx \\
  &=&\int_{\mathbb{R}}\left( \frac{\sin \theta _{1}-\sin \theta _{2}}{
      \sqrt{ \left(  1-\sin ^{2}\theta ^{t} \right) ^{3} } } \, \theta
      ^{t}_x \sin 2 \theta
      ^{t} +\frac{\theta _{1x}\cos
      \theta _{1}-\theta _{2x}\cos \theta _{2}}{\sqrt{1-\sin
      ^{2}\theta ^{t}}} \right) ^{2}dx \\ 
  &&+\int_{\mathbb{R}}\frac{\left( t\theta _{1x}\cos \theta _{1}+\left(
     1-t\right) \theta _{2x}\cos \theta _{2}\right) ^{2}\left( \sin \theta
     _{1}-\sin \theta _{2}\right) ^{2}}{\left( 1-\sin \theta
     ^{t}\right) ^{2}} \, 
     dx \\
  && +\int_{\mathbb{R}}\left( \sin \theta _{1}-\sin \theta _{2}\right) ^{2}dx
     +\frac{\nu }{2}\int_{\mathbb{R}}\left( \sin \theta _{1}-h\right) \left( -%
     \frac{d^{2}}{dx^{2}}\right) ^{1/2}(\sin \theta _{1}-h) \, dx \\
  &&+\frac{\nu }{2}\int_{\mathbb{R}}\left( \sin \theta _{2}-h\right) \left( -%
     \frac{d^{2}}{dx^{2}}\right) ^{1/2}(\sin \theta _{2}-h) \, dx \\
  & > & 0.
\end{eqnarray*}
\end{proof}

\textbf{Proof of Theorem \ref{our}. }Existence and smoothness of solutions
follows from Theorem 1 in \cite{CM}. {{We argue by contradiction and assume
that $\theta _{1}\not\equiv \theta _{2}$}} are two monotone decreasing
solutions of $\left( \ref{thetaequation}\right) $ satisfying $\left( \ref{bc}%
\right) $, together with $E(\theta _{i})<\infty $ and $\theta _{i}\left(
0\right) =\frac{\pi }{2}$. Let $\theta ^{t}$ be defined by $\left( \ref%
{thetat}\right) $ and let $f\left( t\right) =E\left( \theta ^{t}\right) $.
Differentiating $\left( \ref{Ethetat}\right) $ at $t=0,$ we get 

\begin{eqnarray}
  f^{\prime } \left( 0\right)  
  &=&\lim_{t\rightarrow 0^{+}}\int_{\mathbb{R}}\frac{\left\vert \theta
      _{x}^{t}\right\vert ^{2}-\theta _{2x}^{2}}{2t}\,dx+\int_{\mathbb{R}}\left(
      \sin \theta _{2}-h\right) \left( \sin \theta _{1}-\sin \theta
      _{2}\right) \, dx
      \notag \\
  &&-\frac{\nu }{2}\int_{\mathbb{R}}\left( \sin \theta _{2}-h\right) \left( -%
     \frac{d^{2}}{dx^{2}}\right) ^{1/2}(\sin \theta _{2}-h) \, dx  \notag \\
  &&+\frac{\nu }{2}\int_{\mathbb{R}}\left( \sin \theta _{1}-h\right) \left( -%
     \frac{d^{2}}{dx^{2}}\right) ^{1/2}(\sin \theta _{2}-h)\, dx. 
\label{fderivativezero}
\end{eqnarray}
By $\left( \ref{dominate1}\right) $ and dominated convergence theorem,
we have
\begin{eqnarray}
  \lim_{t\rightarrow 0^{+}}\int_{\mathbb{R}}\frac{\left\vert \theta
  _{x}^{t}\right\vert ^{2}-\theta _{2x}^{2}}{2t}\,dx 
  &=&\int_{\mathbb{R}%
      }\lim_{t\rightarrow 0^{+}}\frac{\left\vert \theta _{x}^{t}\right\vert
      ^{2}-\theta _{2x}^{2}}{2t}\,dx  \notag \\
  &=&\int_{\mathbb{R}}\theta _{2x}\left. \theta _{xt}^{t}
      \right\vert _{t=0}dx  = -\int_{\mathbb{R}}\theta
      _{2xx}\left. \theta _{t}^{t} 
      \right\vert _{t=0}dx.  \label{gderivativezero}
\end{eqnarray}%
Here we used the fact that 
\begin{equation*}
\left. \theta _{xt}^{t}\left( x\right) \right\vert _{t=0}=\left\{ 
\begin{array}{cc}
\frac{\theta _{1x}\cos \theta _{1}-\theta _{2x}\cos \theta _{2}}{\cos \theta
_{2}}+\frac{\sin \theta _{2}\left( \sin \theta _{1}-\sin \theta _{2}\right) 
}{\cos ^{2}\theta _{2}}\theta _{2x} & x\neq 0, \\ 
-\frac{\theta _{2x}^{2}\left( 0\right) -\theta _{1x}^{2}\left( 0\right) }{%
2\theta _{2x}\left( 0\right) } & x=0,%
\end{array}%
\right. 
\end{equation*}%
and 
\begin{equation}
\left. \theta _{t}^{t}\left( x\right) \right\vert _{t=0}=\left\{ 
\begin{array}{cc}
\frac{\sin \theta _{1}-\sin \theta _{2}}{\cos \theta _{2}} & x\neq 0, \\ 
0 & x=0,%
\end{array}%
\right.   \label{thetatt}
\end{equation}%
are both continuous on $\mathbb{R}$, which follows from
\eqref{sinlimit}, and
$\frac{\sin \theta _{1}-\sin \theta _{2}}{\cos \theta _{2}}$
approaches zero at infinity. We conclude from
$\left( \ref{fderivativezero}\right) $,
$\left( \ref{gderivativezero}\right) $ and
$\left( \ref{thetatt}\right) $ that
\begin{multline}
  f^{\prime } \left( 0\right) =\int_{\mathbb{R}}\bigg\{ -{d^2 \theta_2
    \over dx^2} +\cos \theta _{2}\left( \sin \theta _{2}-h\right) \\
  +\frac{\nu }{2}\cos \theta _{2}\left( -\frac{d^{2}}{
      dx^{2}}\right) ^{1/2}\sin \theta _{2}\bigg\}
  \left. \frac{d\theta ^{t}}{dt}\left( x\right) \right\vert _{t=0} \,
  dx =0.  \label{frightzero}
\end{multline}%
A similar arguement gives 
\begin{equation}
 f^{\prime } \left( 1\right) =0.  \label{fleftone}
\end{equation}%
On the other hand, it follows from Lemma \ref{Ederivative} that
$f\left( t\right) $ $\in C^{2}\left[ 0,1\right] $ and
$f^{\prime \prime }\left( t\right) >0$ on $\left[
  0,1\right]$. Therefore, one cannot have $\left( \ref%
  {frightzero}\right) $
and $\left( \ref{fleftone}\right) $ to hold at the same time, a
contradiction. \qed
\begin{remark}
  Our proof of uniqueness works as long as $\theta \left( x\right) $
  has range $(\theta_h, \pi-\theta_h)$, satisfies (\ref{bc}) and
  passes through $%
  \frac{\pi }{2}$ only once.
\end{remark}

\section{Uniform bounds and decay of the derivatives}

\subsection{Uniform bound for solutions with bounded energy}

Let 
\begin{equation}
  u\left( x\right) =\sin \theta \left( x\right) -h,\qquad v\left( x\right)
  =\left( -\frac{d^{2}}{dx^{2}}\right) ^{1/2}\sin \theta =\left( - 
    \frac{d^{2}}{dx^{2}}\right) ^{1/2}u(x).  \label{uvdef}
\end{equation}%
We first recall from the proof in section 5, step 2, in \cite{CM} that
any solution $\theta $ of $\left( \ref{thetaequation}\right) $ with
bounded energy is smooth. We shall use this fact for the rest of the
section.

\begin{lemma}
\label{vestimate}Let $\nu >0$, let $h\in \mathbb{R}$ and let $\theta $ be a
solution of \eqref{thetaequation} such that $E\left( \theta \right) <\infty $%
. Then there exists a constant $C=C\left( \nu ,h,E\left( \theta \right)
\right) >0$ such that $\left\vert v\left( x\right) \right\vert \leq C$ for
all $x\in \mathbb{R}$.
\end{lemma}

\begin{proof}

Using the identity 
\begin{equation*}
  \left( -\frac{d^{2}}{dx^{2}}\right) ^{1/2}u(x)=\frac{1}{\pi }\,\text{
    p.v.}\int_{\mathbb{R}}\frac{u\left( x\right) -u\left( y\right) }{\left(
      x-y\right) ^{2}}\, dy  ,
\end{equation*}
for every $x\in \mathbb{R}$ and
$u\in C^{\infty }\left( \mathbb{R}\right) \cap L^{\infty
}(\mathbb{R})$,
where p.v. stands for the principal value of the integral, we can
write
\begin{equation*}
v\left( x\right) =\frac{1}{\pi }\text{ p.v.}\int_{\mathbb{R}}\frac{\sin
\theta \left( x\right) -\sin \theta \left( y\right) }{\left( x-y\right) ^{2}}%
dy.
\end{equation*}%
Given $\delta >0,$ we have
\begin{eqnarray}
\pi v\left( x\right) &=&\text{p.v.}\int_{\mathbb{R}}\frac{\sin \theta \left(
x\right) -\sin \theta \left( y\right) }{\left( x-y\right) ^{2}}dy  \notag \\
&=&\int_{x+\delta }^{\infty }\frac{\sin \theta \left( x\right) -\sin \theta
\left( y\right) }{\left( x-y\right) ^{2}}dy+\int_{-\infty }^{x-\delta }\frac{%
\sin \theta \left( x\right) -\sin \theta \left( y\right) }{\left( x-y\right)
^{2}}dy  \notag \\
&&+\,\text{p.v.}\int_{x-\delta }^{x+\delta }\frac{\sin \theta \left(
x\right) -\sin \theta \left( y\right) }{\left( x-y\right) ^{2}}dy.
\label{vexpression}
\end{eqnarray}%
The first two terms are bounded by $\frac{2}{\delta }$ after direct
integration. Since $\theta $ is smooth, it follows from Taylor expansion
that the third term on the right hand side of $\left( \ref{vexpression}%
\right) $ can be bounded by 
\begin{equation}
\left\vert \text{p.v.}\int_{x-\delta }^{x+\delta }\frac{\sin \theta \left(
x\right) -\sin \theta \left( y\right) }{\left( x-y\right) ^{2}}dy\right\vert
\leq \max_{\left[ x-\delta ,x+\delta \right] }\left\vert u_{xx}\right\vert
\delta .  \label{pvestimate}
\end{equation}%
Furthermore, we have 
\begin{eqnarray}
\max_{\left[ x-\delta ,x+\delta \right] }\left\vert u_{xx}\right\vert
&=&\max_{\left[ x-\delta ,x+\delta \right] }\left\vert \theta _{xx}\cos
\theta -\theta _{x}^{2}\sin \theta \right\vert  \notag \\
&\leq &\max_{\left[ x-\delta ,x+\delta \right] }\left\vert \theta
_{xx}\right\vert +\max_{\left[ x-\delta ,x+\delta \right] }\left\vert \theta
_{x}^{2}\right\vert .  \label{utwoderivative}
\end{eqnarray}%

To estimate the first term on the right-hand side of $\left( \ref%
{utwoderivative}\right) ,$ we use $\left( \ref{thetaequation}\right) $ to
obtain 
\begin{equation}
\max_{\left[ x-\delta ,x+\delta \right] }\left\vert \theta _{xx}\right\vert
\leq 1+|h|+\frac{\nu }{2}\max_{\left[ x-\delta ,x+\delta \right] }\left\vert
v\left( x\right) \right\vert .  \label{thetatwoderivative}
\end{equation}%
To obtain a bound on $\theta _{x}$, we observe that since
$E\left( \theta \right) <\infty ,$ there exists a sequence
$\left\{ x_{n}\right\} \rightarrow -\infty $ such that
$\theta _{x}\left( x_{n}\right) $ $%
\rightarrow 0$.
Therefore, multiplying $\left( \ref{thetaequation}\right) $ by
$\theta _{x}$ and integrating from $x_{n}$ to $x,$ we get
\begin{equation}
\left. \frac{1}{2}\theta _{x}^{2}\left( x\right) \right\vert
_{x_{n}}^{x}=\left. \frac{1}{2}u^{2}\left( x\right) \right\vert _{x_{n}}^{x}+%
\frac{\nu }{2}\int_{x_{n}}^{x}v(y)\,du(y). \label{thetaxsq}
\end{equation}%
Since  
\begin{equation*}
  \int_{x_{n}}^{x}v^{2} \, dx\leq \int_{\mathbb{R}}v^2 \, dx =
  \int_{\mathbb{R}}u_{x}^2 \, dx \leq\int_{\mathbb{R}}\theta_x^2 \, dx
  \leq 2 E(\theta),
\end{equation*}
we can bound the integral in the second term on the right hand side of
(\ref{thetaxsq}) as follows
\begin{eqnarray}
  \left\vert \int_{x_{n}}^{x}v(y)\,du(y)\right\vert 
  \leq \left(
  \int_{x_{n}}^{x}v^{2}\,dx\right) ^{\frac{1}{2}}\left(
  \int_{x_{n}}^{x}  u_{y}^{2}dy \right) ^{\frac{1}{2}} 
  \leq 2 E\left( \theta \right).
\end{eqnarray}
Furthermore, since $\left\vert u\left( x\right) +h\right\vert \leq 1$,
we get
\begin{equation*}
  \theta _{x}^{2}(x)-\theta _{x}^{2}(x_{n}) \leq  (1 +
  |h|)^2+2 \nu E\left( \theta \right).
\end{equation*}%
Finally, sending $n\rightarrow \infty $, we obtain 
\begin{equation}
  \left\vert \theta _{x}\left( x\right) \right\vert \leq \sqrt{(1+|h|)^{2}+%
    {2 \nu }E\left( \theta \right) }\text{ for any }x\in \mathbb{R}\text{%
    .}  \label{thetaderivative}
\end{equation}

From $\left( \ref{vexpression}\right) ,\left( \ref{pvestimate}\right) ,$ $%
\left( \ref{utwoderivative}\right) $, $\left( \ref{thetatwoderivative}%
\right) $ and $\left( \ref{thetaderivative}\right) ,$ we thus conclude 
\begin{equation*} 
  \pi \max_{\mathbb{R}}\left\vert v\right\vert \leq
  \frac{2}{\delta }+\left( 1+\left\vert h\right\vert +\frac{\nu
    }{2}\max_{\mathbb{R}}\left\vert v\right\vert + {2 \nu}
    E\left( \theta \right) +  \left( 1+\left\vert h\right\vert \right)
    ^{2}\right) \delta .
\end{equation*}%
Choosing $\delta =\frac{\pi }{\nu },$ we get
\begin{equation*}
  \max_{\mathbb{R}}\left\vert v\right\vert \leq \frac{4\nu }{\pi
    ^{2}}+\frac{ 2 }{\nu }\left( 1+\left\vert h\right\vert +  \left( 1+\left\vert
        h\right\vert \right) ^{2}\right) + 4 E\left(
    \theta \right).
\end{equation*}
\end{proof}

\begin{corollary}
\label{thetaestimate}There exists $C_{i}=C_{i}\left( \nu ,h ,E\left( \theta
\right) \right) > 0$ $\left( i=1,2,\ldots ,\right) $ such that, given any
solution $\theta $ of $\left( \ref{thetaequation}\right) $ with $%
E\left( \theta \right) < \infty$, we have 
\begin{equation*}
\sup_{x\in \mathbb{R}}\left( \left\vert \frac{d^{i}\theta }{dx^{i}}\left(
x\right) \right\vert \right) \leq C_{i}.
\end{equation*}
\end{corollary}

\begin{proof}
The estimate for $\theta _{x},$ $\theta _{xx}$ follows directly from $\left( %
\ref{thetatwoderivative}\right) ,$ $\left( \ref{thetaderivative}\right) $
and Lemma \ref{vestimate}. To estimate $\theta _{xxx},$ differentiate $%
\left( \ref{thetaequation}\right)$. We have 
\begin{equation*}
  \theta _{xxx}=\theta _{x}\cos ^{2}\theta -\theta _{x}\sin \theta \left( \sin
    \theta -h\right) -v\theta _{x}\sin \theta +v_x \cos \theta.
\end{equation*}%
It then follows that
\begin{equation*}
  \left\vert \theta _{xxx}\right\vert \leq C+\left\vert v_x
  \right\vert .
\end{equation*}
Since 
\begin{equation*}
  \pi v_x \left( x\right) =\text{p.v.}\int \frac{u _{x}\left(
        x\right) - u_{y}\left( y\right) } {\left( x-y\right) ^{2}}dy,
\end{equation*}%
using Taylor expansion, we get 
\begin{equation*}
\pi \left\vert v_x \right\vert \leq \max \left\vert u_{ xxx} \right\vert
\delta +\frac{C}{\delta }.
\end{equation*}%
On the other hand, 
\begin{equation*}
u_{xxx} =\theta _{xxx}-\theta _{x}^{3}\cos \theta -3\theta_{x}\theta
_{xx}\sin \theta.
\end{equation*}%
We can then follow a similar argument as in Lemma \ref{vestimate} to
get a bound on $\left\vert v_{x}\right\vert $ and, thus, a bound on
$\left\vert \theta _{xxx}\right\vert $. Differentiating repeatedly, we
obtain similar estimates for all derivatives.
\end{proof}

Since any solution of (\ref{thetaequation}) with bounded energy is in
$W^{k,2}(\mathbb R)$ for any $k \in \mathbb N$, as a direct
corollary of our bound on the derivatives of $\theta$ we conclude that
any solution of $\left( \ref{thetaequation}\right) $ with bounded
energy must have all its derivatives vanish at infinity. \bigskip

\begin{corollary}
  \label{firstderivative}If $\theta$ is a solution of
  $\left( \ref{thetaequation}\right) $ with bounded energy$,$ then all
  the derivatives of $\theta$ vanish at infinity.
\end{corollary}

\noindent The proof of theorem \ref{our} is now complete, combing
corollaries \ref{thetaestimate} and \ref{firstderivative}. \qed

\appendix

\section{Decay of N\'eel walls}

In this section, we revisit the question of the asymptotic decay of
N\'eel wall solutions, whose existence and uniqueness is guaranteed by
Theorem \ref{cm}. Let $\theta$ be the unique minimizer of $E$ in
$\mathcal A$ satisfying $\theta(0) = \frac{\pi}{2}$, and introduce
\begin{align}
  \label{eq:1}
  \rho(x) := 
  \begin{cases}
    \theta(x), & x \geq 0, \\
    \pi - \theta(x), & x < 0.
  \end{cases}
\end{align}
Note that $\rho - \theta_h \in H^1(\mathbb R)$,
$\sin \rho = \sin \theta$, $\cos \rho = \text{sgn}(x) \cos \theta$,
and $\rho(x)$ is a smooth even function of $x$, except at $x = 0$,
where $\rho_x$ undergoes a jump discontinuity.

Proceeding as in Step 4 of the proof of Theorem 1 in \cite{CM}, we
observe that $\rho$ satisfies distributionally
\begin{align}
  \label{eq:2}
  L( \rho(x) - \theta_h) = f(x) + a \delta(x),
\end{align}
where 
\begin{align}
  \label{eq:3}
  L := -{d^2 \over dx^2} + \frac12 \nu \cos^2 \theta_h \left( -{d^2
  \over dx^2} \right)^{1/2} + \cos^2 \theta_h
\end{align}
is a one-to-one linear map from $\mathcal S'(\mathbb R)$ to itself,
which is also a bounded operator from $H^2(\mathbb R)$ to
$L^2(\mathbb R)$, $f \in L^2(\mathbb R)$ is defined in Eq.~(66) of
\cite{CM}, $a = 2 |\theta'(0)| > 0$ and $\delta(x)$ is the Dirac
delta-function. The last term in the right-hand side of \eqref{eq:2}
was inadvertently omitted in \cite{CM}. Nevertheless, its presence
does not affect the rest of the proof. Namely, we invert the operator
$L$ with the help of the fundamental solution $G$ (see \cite[Lemma
A.1]{CM} for an explicit definition and properties of $G$). In
particular, we have
\begin{align}
  \label{eq:4}
  \rho(x) = \theta_h + 2 a G(x) + \int_\mathbb{R} G(x - y) \, f(y) \, dy,
\end{align}
for every $x \in \mathbb R$. The presence of the $2 a G$ term in the
right-hand side of \eqref{eq:4} leaves all the remaining estimates
unchanged, in view of the fact that $0 < G(x) \leq C / |x|^2$, for
some $C > 0$.

\end{document}